\documentclass[11pt]{paper}
\usepackage{amssymb}
\usepackage{amsmath}
\usepackage{amsthm}
\usepackage{array}
\usepackage{graphicx}
\usepackage{subfigure}
\usepackage{epstopdf}
\usepackage{cases}
\usepackage{color}
\usepackage{multirow}
\usepackage{picinpar}
\usepackage[colorlinks,
            linkcolor=red,
            anchorcolor=green,
            citecolor=blue
            ]{hyperref}
\usepackage{enumitem}
\usepackage{fourier}
\usepackage{bookmark}
\begin{document}

\newcommand{\defi}{\stackrel{\Delta}{=}}
\newcommand{\A}{{\cal A}}
\newcommand{\B}{{\cal B}}
\newcommand{\U}{{\cal U}}
\newcommand{\G}{{\cal G}}
\newcommand{\cZ}{{\cal Z}}
\newcommand\one{\hbox{1\kern-2.4pt l }}
\newcommand{\Item}{\refstepcounter{Ictr}\item[\left(\theIctr\right)]}
\newcommand{\QQ}{\hphantom{MMMMMMM}}

\newtheorem{Theorem}{Theorem}[section]
\newtheorem{Lemma}{Lemma}[section]
\newtheorem{Corollary}{Corollary}[section]
\newtheorem{Remark}{Remark}[section]
\newtheorem{Example}{Example}[section]
\newtheorem{Proposition}{Proposition}[section]
\newtheorem{Property}{Property}[section]
\newtheorem{Assumption}{Assumption}[section]
\newtheorem{Definition}{Definition}[section]
\newtheorem{Construction}{Construction}[section]
\newtheorem{Condition}{Condition}[section]
\newtheorem{Exa}[Theorem]{Example}
\newcounter{claim_nb}[Theorem]
\setcounter{claim_nb}{0}
\newtheorem{claim}[claim_nb]{Claim}
\newenvironment{cproof}
{\begin{proof}
 [Proof.]
 \vspace{-3.2\parsep}}
{\renewcommand{\qed}{\hfill $\Diamond$} \end{proof}}
\newcommand{\erhao}{\fontsize{21pt}{\baselineskip}\selectfont}
\newcommand{\xiaoerhao}{\fontsize{18pt}{\baselineskip}\selectfont}
\newcommand{\sanhao}{\fontsize{15.75pt}{\baselineskip}\selectfont}
\newcommand{\sihao}{\fontsize{14pt}{\baselineskip}\selectfont}
\newcommand{\xiaosihao}{\fontsize{12pt}{\baselineskip}\selectfont}
\newcommand{\wuhao}{\fontsize{10.5pt}{\baselineskip}\selectfont}
\newcommand{\xiaowuhao}{\fontsize{9pt}{\baselineskip}\selectfont}
\newcommand{\liuhao}{\fontsize{7.875pt}{\baselineskip}\selectfont}
\newcommand{\qihao}{\fontsize{5.25pt}{\baselineskip}\selectfont}
\newcounter{Ictr}
\renewcommand{\theequation}{
\arabic{equation}}
\renewcommand{\thefootnote}{\fnsymbol{footnote}}

\def\A{\mathcal{A}}

\def\C{\mathcal{C}}

\def\V{\mathcal{V}}

\def\I{\mathcal{I}}

\def\Y{\mathcal{Y}}

\def\X{\mathcal{X}}

\def\J{\mathcal{J}}

\def\Q{\mathcal{Q}}

\def\W{\mathcal{W}}

\def\S{\mathcal{S}}

\def\T{\mathcal{T}}

\def\L{\mathcal{L}}

\def\M{\mathcal{M}}

\def\N{\mathcal{N}}
\def\R{\mathbb{R}}
\def\H{\mathbb{H}}

\title{}
\author{}
\begin{center}
\topskip2cm
\LARGE{\bf Regularization parameter selection for low rank matrix recovery}
\end{center}

\begin{center}
\renewcommand{\thefootnote}{\fnsymbol{footnote}}Pan Shang,  Lingchen Kong, \footnote{e-mail: 18118019@bjtu.edu.cn,   konglchen@126.com, }\\
Beijing Jiaotong University, China\\
(September 16th, 2019)\\
\end{center}
\vskip4pt
\begin{abstract}
Low rank matrix recovery is the focus of many applications, but it is a NP-hard problem.  A popular way to deal with this problem is to solve its convex relaxation, the nuclear norm regularized minimization problem (NRM), which includes LASSO as a special case. There are some regularization parameter selection results for LASSO in vector case, such as screening rules, which improve the efficiency of the algorithms.  However, there are no corresponding parameter selection results for NRM in matrix case.  In this paper, we build up a novel rule to choose the regularization parameter for NRM under the help of duality theory. This rule claims that the regularization parameter can be easily chosen by feasible points of NRM and its dual problem, when the rank of the desired solution is no more than a given constant. In particular, we apply this idea to NRM with least square and Huber functions, and establish the easily calculated formula of regularization parameters. Finally, we report numerical results on some signal shapes, which state that our proposed rule shrinks the interval of the regularization parameter efficiently.
\end{abstract}

\section{Introduction}
\label{intro}
Low rank matrix recovery problem arises in tremendous applications, such as machine learning \cite{C12}, signal process \cite{M16}, system identification \cite{L09}, biomedical imaging \cite{Z10} and so on, but it is a NP-hard problem.  A popular way to recovery the low rank matrix is to solve the nuclear norm minimization problem (e.g., Fazel \cite{F02}), which is its convex relaxation. Actually, the nuclear norm minimization problem can yield the exact solution of the low rank matrix recovery problem under some assumptions, such as the restricted isometry property (RIP, e.g., Recht et al. \cite{R10}, Cai and Zhang \cite{C13}) and the s-goodness condition (e.g., Kong et. al. \cite{K13}). In statistics and machine learning areas, we usually consider the corresponding unconstrained optimization problem, which is called the nuclear norm regularized minimization problem (NRM) (see, e.g., Koltchinskii et al. \cite{K11}, Mark and Justin \cite{M16}, Bottou and Nocedal \cite{B18},  Negahban and Wainwright \cite{N11}, Rohde and Tsybakov \cite{R11}, Zhou and Li \cite{Z14}).

Regularization parameter selection plays an essential role for solving the regularization model and cross validation is a common method to choose this parameter in statistics and machine learning. 
It is well known that there are some screening rules for LASSO in vector case, that help to choose the regularization (or tuning) parameter. See, e.g.,  Fan and Lv \cite{F08}, Ghaoui et al. \cite{G12},  Tibshirani et al. \cite{T12}, Wang et al. \cite{W15}, Eugene et al. \cite{E17}, Kuang et al. \cite{K17}, Xiang et al. \cite{X17}, Lee et al. \cite{L18}. For instance, Ghaoui et al. \cite{G12} constructed SAFE rules to eliminate predictors and these rules never remove active predictors. Tibshirani et al. \cite{T12} proposed strong rules for discarding inactive predictors under the unit slope bound assumption.  The strong rules screen out far more predictors than SAFE rules in practice and can be more effective by checking Karush-Kuhn-Tucker (KKT) conditions for any predictor.  Eugene et al. \cite{E17} built up  statics and dynamic  gap safe screening rules for LASSO which are based on the gap between feasible points of LASSO and its dual problem.  In the sense of the sparse solution of LASSO, screening rules can be applied to choose the regularization parameter.  That is, screening rules imply the parameter selection approach of LASSO which guarantee the sparsity of the solution of LASSO is no more than a given constant. However, to the best of our knowledge, there are no regularization parameter selection results for NRM from optimal perspective. Note that NRM includes sparse vector selection (compress sensing) as a special case. Thus, NRM degrades into LASSO type problems when the unknown variable is vector. One nature question occurs: can we establish the regularization parameter selection result for NRM?

In this paper, we give an affirmative answer and build up the regularization parameter selection rule for NRM. In order to do so, we present the dual form of NRM and establish the strong duality theorem.  With the help of duality theory,  we can obtain the  regularization parameter selection rule for NRM based on its dual solution, when the rank of the desired solution of NRM is no more than a constant. This is the primal result that can be used to select the regularization parameter, but it may be a complex work to get the dual solution. Furthermore, by analyzing the dual problem of NRM, we obtain a novel regularization parameter selection rule, which depends on feasible points of primal and  dual problems. Moreover, this idea is applied to the nuclear norm regularized least square minimization (LS-NRM) and  the nuclear norm regularized Huber minimization (H-NRM), respectively. For every problem, the regularization parameter selection rule gives a sequence of closed-form parameters and the rank of the solution is bounded in these intervals. For the purpose of enlightening the regularization parameter selection rule, we consider signal shapes in Zhou and Li \cite{Z14}. Numerical results show that the interval of the regularization parameter can be shrunken when the rank of the solution is given.

The rest of the paper is organized as follows. We review some related models and build up the duality theory of the general model NRM in Section 2.  In Section 3, we show the general idea of the regularization parameter selection rule for NRM. We apply the regularization parameter selection rule to LS-NRM and H-NRM, respectively.  In Section 5, we present the numerical results of the regularization parameter selection rule.  Some conclusions are given in Section 6.

\emph{Notations}: For vector $\textbf{x}\in R^{n}$, the 2-norm $\|\cdot\|_{2}$ is defined as $\|\textbf{x}\|_{2}=\sqrt{\sum_{i=1}^{n}x_{i}^{2}}$. For any matrix $M\in R^{p\times q}$, suppose $M$ has a singular value decomposition with nonincreasing singular values  $\sigma_{1}(M)\geq \cdots \sigma_{r}(M)\geq 0$ where $r=\rm{min\left\{p, q\right\}}$. There are some norms based on singular values of $M$. The Frobenius norm $\|\cdot\|_{F}$ is defined as $\|M\|_{F}=\sqrt{\sum_{i=1}^{p}\sum _{j=1}^{q}M_{ij}^{2}}=\sqrt{\sigma^{2}_{1}(M)+\cdots+\sigma^{2}_{r}(M)}$. The nuclear norm $\|\cdot\|_{*}$ is  the sum of singular values, i.e., $\|M\|_{*}=\sum_{i=1}^{r}{\sigma_{i}(M)}$. The spectral norm $\|\cdot\|_{2}$ is the largest singular value, i.e., $\|M\|_{2}=\sigma_{1}(M)$.
\section{Preliminary}
\label{sec:2}
In this section, we review the low rank matrix recovery problem (LRM) and its convex relaxation. In particular, we analyze the nuclear norm regularized minimization problem (NRM) and build up its duality theory.

The low rank matrix recovery problem is to find a low rank matrix $B$ which satisfies some linear constraints, that is
\begin{equation*}
\begin{split}
&\underset{B\in R^{p\times q}}\min {\rm{rank} (B)}\\
&s.t. \quad \mathcal{Y}=\mathcal{A}(B)+\epsilon, \quad \|\epsilon\|_{2}\leq\delta,
\end{split}
\end{equation*}
where $\mathcal{A}(B)=\left(\langle X_{1},B\rangle,\cdots,\langle X_{n},B\rangle\right)^{T}$, $(X_{i},y_{i})\in R^{p\times q}\times R$ ($i=1,2,\cdots,n$) are given, $\epsilon$ is the noise vector and $\delta\geq 0$ is a constant. If $\delta=0$, the model is the noiseless LMR. Otherwise, it is the noise LMR. It is a NP-hard problem because rank$(B)$ is noncontinous and nonconvex. However, LMR has many important applications in statistics, machine learning and so on. See, e.g., \cite{C12},  \cite{M16},  \cite{L09}, \cite{Z10}.

A popular technique is to solve LMR via its convex relaxation, which is the following well-known nuclear norm minimization problem  (e.g., Fazel \cite{F02})
\begin{equation*}
\begin{split}
&\underset{B\in R^{p\times q}}\min  \quad{||B||_{*}}\\
&s.t. \quad \mathcal{Y}=\mathcal{A}(B)+\epsilon,  \quad \|\epsilon\|_{2}\leq\delta.
\end{split}
\end{equation*}
Note that there are many researches on the theoretical guarantee that this relaxation problem yields the exact solution of the low rank matrix recovery problem under some conditions, such as the restricted isometry property (see, e.g., Recht et al. \cite{R10}, Cai and Zhang \cite{C13}) and  s-goodness condition (e.g., Kong et al.\cite{K11}). There are much attention on methods to solve this convex relaxation problem. One popular model is the nuclear norm regularized least square minimization (LS-NRM). See (7) in Section 4. In statistics, when the noise $\epsilon$ has zero mean and constant variance, LS-NRM has good performance to recovery the low rank matrix. However, it is not efficient when the noise has heavy tails or outlier. In fact, the distribution of noise is not known in practice. In this sense, another recently attractive model is the nuclear norm regularized Huber minimization (H-NRM, see, e.g., Huber \cite{H73}, Sun \cite{S17}, Elsener and Geer \cite{E18} and Chen et al.\cite{C}). See (10) in Section 4.

How to choose the regularization parameter for nuclear norm regularized minimization problem is an essential question. In literatures of LS-NRM and H-NRM, cross validation is often used to select this parameter.  To the best of our knowledge, there are no regularization parameter selection theoretical results from the point of optimization. In order to establish this kind of result, we introduce a general nuclear norm regularized minimization problem (NRM) as follows,
\begin{eqnarray}
\underset{B\in R^{p\times q}}\min \left\{ F_{\lambda}(B)=\sum\limits_{i=1}^{n}f_{i}\left(y_{i}-\langle X_{i},B\rangle\right)+\lambda||B||_{*}\right\},
\end{eqnarray}
where $f_{i}: R\mapsto \overline{R}$ is a proper, closed and convex function with $\frac{1}{\alpha}$-Lipschitz continuous gradient ($\alpha>0$). Obviously, this problem includes LS-NRM and H-NRM as special cases. In order to address that the solution of problem (1) depends on the regularization parameter $\lambda$, we denote it as $B^{*}(\lambda)$.

Duality theory plays an important role in building up the regularization parameter selection results. Thus, we consider the dual problem of NRM (1). By introducing new variables $t_{i}= y_{i}-\langle X_{i},B\rangle$, $i=1,2,\cdots,n$, we rewrite problem (1) as
\begin{equation}
\begin{split}
&\underset{B\in R^{p\times q},\textbf{t}\in R^{n}}\min\left\{ F_{\lambda}(B,\textbf{t})=\sum\limits_{i=1}^{n}f_{i}(t_{i})+\lambda||B||_{*}\right\}\\
&s.t.\quad y_{i}-\langle X_{i},B\rangle-t_{i}=0, \quad i=1,2,\cdots, n,
\end{split}
\end{equation}
where $\textbf{t}=(t_{1},t_{2},\cdots,t_{n})^{T}$. Then, the  Lagrangian function of  problem (2) is
\begin{center}
$\textit{L}\left(B,\textbf{t};\theta \right)=\sum\limits_{i=1}^{n}f_{i}(t_{i})+\lambda||B||_{*}
+\sum\limits_{i=1}^{n}\theta_{i}\cdot\left(y_{i}-\langle X_{i},B\rangle-t_{i}\right)$,
\end{center}
where $\theta=(\theta_{1},\theta_{2},\cdots,\theta_{n})^{T}$ with $\theta_{i}\in R$ $ (i=1,2,\cdots,n)$ being the Lagrangian multiplier of problem (2).   By direct computation, we obtain that
\begin{align*}
&\underset{B\in R^{p\times q},\textbf{t}\in R^{n}}\min{\textit{L}\left(B,\textbf{t};\theta\right)}\\
&=\underset{B\in R^{p\times q}}\min\left\{\lambda||B||_{*}-\left\langle\sum\limits_{i=1}^{n}\theta_{i} X_{i},B\right\rangle\right\}
+\underset{\textbf{t}\in R^{n}}\min\left\{\sum\limits_{i=1}^{n}{\left(f_{i}(t_{i})-\theta_{i}t_{i}\right)}\right\}+\langle\textbf{y},\theta\rangle\\
&=-\underset{B\in R^{p\times q}}\max\left\{\left\langle\sum\limits_{i=1}^{n}\theta_{i} X_{i},B\right\rangle-\lambda||B||_{*}\right\}
-\sum\limits_{i=1}^{n}\underset{t_{i}\in R}\max\left\{\theta_{i}t_{i}-f_{i}(t_{i})\right\}+\langle\textbf{y},\theta\rangle.
\end{align*}
Based on the definition of the conjugate function (see, Rockafellar \cite{R70}), we know that
\begin{eqnarray*}
\underset{B\in R^{p\times q}}\max\left\{\left\langle\sum\limits_{i=1}^{n}\theta_{i} X_{i},B\right\rangle-\lambda||B||_{*}\right\}=
\begin{cases}
0, &\left\|\sum\limits_{i=1}^{n}\theta_{i} X_{i}\right\|_{2}\leq\lambda,\\
+\infty, &\rm{otherwise},
\end{cases}
\end{eqnarray*}
and $\underset{t_{i}\in R}\max\left\{\theta_{i}t_{i}-f_{i}(t_{i})\right\}=f^{*}_{i}(\theta_{i})$. Thus,
\begin{center}
$\underset{B\in R^{p\times q},\textbf{t}\in R^{n}}\min{\textit{L}\left(B,\textbf{t};\theta\right)}
=\begin{cases}
-\sum\limits_{i=1}^{n}f^{*}_{i}(\theta_{i})+\langle\textbf{y},\theta\rangle, &\left\|\sum\limits_{i=1}^{n}\theta_{i} X_{i}\right\|_{2}\leq\lambda,\\
-\infty, &\rm{otherwise}.
\end{cases}
$
\end{center}
Therefore, the dual problem of (2) is
\begin{equation}
\begin{split}
&\underset{\theta \in R^{n}}\max \left\{ G_{\lambda}(\theta)=\langle \textbf{y}, \theta\rangle-\sum\limits_{i=1}^{n}f^{*}_{i}(\theta_{i})\right\}\\
&s.t.\quad\left\|\sum\limits_{i=1}^{n} \theta_{i}X_{i}\right\|_{2}\leq \lambda.\\
\end{split}
\end{equation}
Denote the 
the solution of (3) as $\theta^{*}(\lambda)$. The Karush-Kuhn-Tucker (KKT) system of (2) and (3) is
\begin{eqnarray}
\begin{cases}
\sum\limits_{i=1}^{n} \theta_{i}X_{i}\in \lambda\partial\|B\|_{*},\\
y_{i}-\langle X_{i},B\rangle-t_{i}=0,\\
\theta_{i}=\nabla f_{i}(t_{i}), \quad i=1,2,\cdots,n.
\end{cases}
\end{eqnarray}
If a pair $\left(B(\lambda), \textbf{t}(\lambda), \theta(\lambda)\right)$ satisfies the KKT system, it is called the KKT point of (2) and (3). Note that $(0,\textbf{y})$ is a feasible point of  problem (2), Slater constraint qualification holds on this problem. Hence, we easily show the following duality theorem.
\begin{Theorem}
(\textbf{Strong duality theorem}) If the solution of problem (2) exists,  then there is a KKT point $(B^{*}(\lambda),\textbf{t}^{*}(\lambda),\theta^{*}(\lambda))$ such that  the optimal values of problems (2) and (3) are equal, i.e.,
$$F_{\lambda}(B^{*}(\lambda),\textbf{t}^{*}(\lambda))=G_{\lambda}(\theta^{*}(\lambda)).$$
Here, $B^{*}(\lambda)$ is the solution of (1) and $\theta^{*}(\lambda)$ is the solution of (3).
\end{Theorem}
\section{Regularization parameter selection rule}
\label{sec:3}
With the help of the duality theory in Section 2,  we will present the general regularization parameter selection rule for NRM in this section, which clarifies the regularization parameter can be selected by feasible points of NRM and its dual problem. When solving NRM with the selected parameter, we can guarantee that the rank of its solution is no more than a given data.

It is clear that the solution of problem (1)  is zero when the regularization parameter $\lambda$ is sufficient large. We wonder the lower bound of the $\lambda$ such that the corresponding solution is zero. The following lemma gives the interesting connection between zero solution of problem (1) and this lower bound.

\begin{Lemma}
Let  $\theta^{max}=\underset{\theta\in R^{n}}{\arg\max}\{\langle \textbf{y}, \theta\rangle-\sum\limits_{i=1}^{n}f^{*}_{i}(\theta_{i})\}$ and $\lambda_{max}=\left\|\sum\limits_{i=1}^{n}\theta_{i}^{max}X_{i}\right\|_{2}$. The following statements hold.\\
 (1) If $B^{*}(\lambda)=0$, then $\lambda\geq\lambda_{max}$.\\
(2) If $\lambda>\lambda_{max}$, then $B^{*}(\lambda)=0$.

\end{Lemma}
\begin{proof}
(1) Since $B^{*}(\lambda)=0$, by Theorem 2.1, we know that $$F_{\lambda}(0,\textbf{y})=G_{\lambda}(\theta^{*}(\lambda)).$$
It is true that $G_{\lambda}(\theta^{*}(\lambda))\leq \underset{\theta\in R^{n}}\max\langle \textbf{y},\theta\rangle-\sum\limits_{i=1}^{n}f^{*}_{i}(\theta_{i})$. The convexity of $\{f_{i}\}_{i=1}^{n}$ implies that $$G(\theta_{max})=\sum\limits_{i=1}^{n}f_{i}(y_{i})=F_{\lambda}(0,\textbf{y}).$$ Combing these results, we know that
$$F_{\lambda}((0,\textbf{y})=G_{\lambda}(\theta^{*}(\lambda))\leq G(\theta_{max})=F_{\lambda}((0,\textbf{y}).$$
Because $f_{i}$ is differentiable and its gradient is  $\frac{1}{\alpha}$-Lipschitz continuous, its conjugate function  $f^{*}_{i}$ is $\alpha$-strongly convex (Hiriart-Urruty and Lemar$\acute{e}$chal \cite{H93}). Then, the function $G_{\lambda}(\theta)$ is $n\alpha$-strongly concave and the problem (3) has an unique solution. Therefore, $\theta^{max}$ is a solution of problem (3) and $\left\|\sum\limits_{i=1}^{n}\theta_{i}^{max}X_{i}\right\|_{2}\leq\lambda$. Hence, the desired result follows.\\
(2) If  $\lambda>\lambda_{max}$, we can get $\left\|\sum\limits_{i=1}^{n}\theta_{i}^{max}X_{i}\right\|_{2}\leq\lambda$ and $\theta^{*}(\lambda)=\theta^{max}$. By Theorem 2.1, we know that $G_{\lambda}(\theta^{max})=F_{\lambda}(B^{*}(\lambda),\textbf{t}^{*}(\lambda))$. Because $(0,\textbf{y})$ is a feasible point of problem (2), it is clear that $F_{\lambda}(B^{*}(\lambda),\textbf{t}^{*}(\lambda))\leq F_{\lambda}(0,\textbf{y})$. According to the convexity of $f_{i}$, $F_{\lambda}(0,\textbf{y})= G_{\lambda}(\theta^{max})$. Therefore, we know that
\begin{center}
$G_{\lambda}(\theta^{max})=F_{\lambda}(B^{*}(\lambda),\textbf{t}^{*}(\lambda))\leq F_{\lambda}(0,\textbf{y})= G_{\lambda}(\theta^{max})$.
\end{center}
Clearly, zero is a solution of NRM (1) when $\lambda>\lambda_{max}$. Moreover, $\sum\limits_{i=1}^{n} \theta^{max}_{i}X_{i}\in \lambda\partial\|B^{*}(\lambda)\|_{*}$. Thus, $B^{*}(\lambda)=0$ is the unique solution. \qed
\end{proof}

From Lemma 3.1, without loss of generality, we focus on the case of the regularization parameter $\lambda$ in $(0,\lambda_{max}]$ with $\lambda_{max}>0$. We next show a basic result which claims the relationship between regularization parameter $\lambda$ and the rank of the solution of NRM (1).
\begin{Theorem}\label{4.2}
For any $k\in\{1,2,\cdots,r\}$, if
$$\lambda>\sigma_{k}\left(\sum\limits_{i=1}^{n} \theta_{i}^{*}(\lambda)X_{i}\right),$$
then $\sigma_{k}(B^{*}(\lambda))=0$. This leads to rank$(B^{*}(\lambda))\leq k-1$.
\end{Theorem}
\begin{proof}
According to the KKT system of problems (2) and (3),  we have
$$\sum\limits_{i=1}^{n} \theta^{*}_{i}(\lambda)X_{i}\in \lambda\partial\|B^{*}(\lambda)\|_{*}.$$
The desired conclusion follows from the formula of $\partial\|
\cdot\|_{*}$.\qed
\end{proof}

The above theorem presents the basic idea to choose the regularization parameter. That is, for any fixed regularization parameter, we may apply the dual solution $\theta^{*}(\lambda)$ to decide whether the  singular value of the $B^{*}(\lambda)$ is zero. Based on this, the rank of the $B^{*}(\lambda)$ can be bounded. Notice that when we apply Theorem  \ref{4.2} to choose the regularization parameter for NRM (1), we first need to get the solution of its dual problem (3). However, the dual solution may be hard to yield. Usually, we can easily obtain a feasible set which contains the dual solution. The following lemma states one feasible set based on the dual gap of problems (2) and (3). Here, $f_{i}$ in problem (2) is differentiable and its gradient is  $\frac{1}{\alpha}$-Lipschitz continuous. Denote the duality gap of problems (2) and (3) as $Gap(\lambda)=F_{\lambda}(B(\lambda),\textbf{t}(\lambda))-G_{\lambda}(\theta(\lambda))$.
\begin{Lemma}\label{4.3}
For any feasible points $(B(\lambda), \textbf{t}(\lambda))$ of problem (2) and $\theta(\lambda)$ of problem (3), it holds that
\begin{eqnarray}
\|\theta(\lambda)-\theta^{*}(\lambda)\|_{2}\leq\sqrt{\frac{2Gap(\lambda)}{n\alpha}}.
\end{eqnarray}
\end{Lemma}
\begin{proof}
 Because $f_{i}$ is differentiable and its gradient is  $\frac{1}{\alpha}$-Lipschitz continuous, its conjugate function  $f^{*}_{i}$ is $\alpha$-strongly convex (Hiriart-Urruty and Lemar$\acute{e}$chal \cite{H93}). So, $G_{\lambda}(\theta)$ is $n\alpha$-strongly concave, i.e.,
\begin{center}
$G_{\lambda}(\theta_{1})\leq G_{\lambda}(\theta_{2})+\langle\nabla G_{\lambda}(\theta_{2}), \theta_{2}-\theta_{1}\rangle-\frac{n\alpha}{2}\|\theta_{1}-\theta_{2}\|^{2}_{2}$
\end{center}
for any $\theta_{1}$, $\theta_{2}\in R^{n}$. Setting $\theta_{1}=\theta(\lambda)$ and $\theta_{2}=\theta^{*}(\lambda)$, we have
\begin{center}
$G_{\lambda}(\theta(\lambda))\leq G_{\lambda}(\theta^{*}(\lambda))+\langle\nabla G_{\lambda}(\theta^{*}(\lambda)), \theta(\lambda)-\theta^{*}(\lambda)\rangle-\frac{n\alpha}{2}\|\theta^{*}(\lambda)-\theta(\lambda)\|^{2}_{2}$.
\end{center}
 It is true that $\langle\nabla G_{\lambda}(\theta^{*}(\lambda)), \theta(\lambda)-\theta^{*}(\lambda)\rangle\leq0$, because $\theta^{*}(\lambda)$ is the solution of problem (3). So, the last inequality can be simplified as
\begin{center}
$G_{\lambda}(\theta(\lambda))\leq G_{\lambda}(\theta^{*}(\lambda))-\frac{n\alpha}{2}\|\theta^{*}(\lambda)-\theta(\lambda)\|^{2}_{2}$.
\end{center}
This, together with $G_{\lambda}(\theta^{*}(\lambda))\leq F_{\lambda}(B(\lambda),\textbf{t}(\lambda))$, suggests that
\begin{center}
$G_{\lambda}(\theta(\lambda))\leq F_{\lambda}(B(\lambda),\textbf{t}(\lambda))-\frac{n\alpha}{2}\|\theta^{*}(\lambda)-\theta(\lambda) \|^{2}_{2}$.
\end{center}
Therefore, the conclusion is proved. \qed
\end{proof}

 Lemma \ref{4.3} means that the dual solution $\theta^{*}(\lambda)$ is contained in the set $$\mathcal{B}=\left\{\theta\Big|\|\theta-\theta(\lambda)\|_{2}\leq\sqrt{\frac{2Gap(\lambda)}{n\alpha}}\right\}.$$
 According to Theorem \ref{4.2}, if
\begin{eqnarray}
\underset{\theta\in \mathcal{B}}\max ~{\sigma_{k}}{\left(\sum\limits_{i=1}^{n} \theta_{i}X_{i}\right)}<\lambda,
\end{eqnarray}
then $\sigma_{k}(B^{*}(\lambda))=0$. Combining Theorem \ref{4.2} and  Lemma \ref{4.3}, we obtain the following regularization parameter selection rule.
\begin{Theorem}\label{4.4}
Let $(B(\lambda), \textbf{t}(\lambda))$ and $\theta(\lambda)$ be any feasible points of problems (2) and (3), respectively. For any $k\in\{1,2,\cdots,r\}$, if
\begin{center}
$\lambda>\sigma_{k}\left(\sum\limits_{i=1}^{n} \theta_{i}(\lambda)X_{i}\right)+\sqrt{\frac{2Gap(\lambda)}{\alpha}\cdot\frac{\sum\limits_{i=1}^{n}\|X_{i}\|_{2}^{2}}{n}}$,
\end{center}
then $\sigma_{k}(B^{*}(\lambda))=0$. This leads to $\rm rank$$(B^{*}(\lambda))\leq k-1$.
\end{Theorem}
\begin{proof}
For any regularization parameter $\lambda$ and feasible point $\theta(\lambda)$ of problem (3), Lemma 3.2 shows that the dual solution $\theta^{*}(\lambda)$ is contained in the set $\mathcal{B}=\left\{\theta\Big|\|\theta-\theta(\lambda)\|_{2}\leq\gamma\right\}$, where
$\gamma=\sqrt{\frac{2Gap(\lambda)}{n\alpha}}$. Setting $\eta=\theta-\theta(\lambda)$, we have
\begin{center}
$\underset{\theta\in \mathcal{B}}\max{~\sigma_{k}}\left(\sum\limits_{i=1}^{n}\theta_{i}X_{i}\right)
=\underset{\|\eta\|_{2}\leq\gamma}\max{\sigma_{k}}\left(\sum\limits_{i=1}^{n}\left(\theta_{i}(\lambda)+\eta_{i}\right)X_{i}\right).$
\end{center}
Following the inequalities of singular value in Roger \cite{R13} (see, page 454), we get
\begin{center}
$\sigma_{k}\left(\sum\limits_{i=1}^{n}\left(\theta_{i}(\lambda)+\eta_{i}\right)X_{i}\right)
\leq \sigma_{k}\left(\sum\limits_{i=1}^{n}\theta_{i}(\lambda)X_{i}\right)
+\sigma_{1}\left(\sum\limits_{i=1}^{n}\eta_{i}X_{i}\right).$
\end{center}
Then
\begin{align*}
\underset{\|\eta\|_{2}\leq \gamma}\max{\sigma_{k}}\left(\sum\limits_{i=1}^{n}\left(\theta_{i}(\lambda)+\eta_{i}\right)X_{i}\right)
&\leq \sigma_{k}\left(\sum\limits_{i=1}^{n}\theta_{i}(\lambda)X_{i}\right)+\underset{\|\eta\|_{2}\leq \gamma}\max{\sigma_{1}}\left(\sum\limits_{i=1}^{n}\eta_{i}X_{i}\right)\\
&=\sigma_{k}\left(\sum\limits_{i=1}^{n}\theta_{i}(\lambda)X_{i}\right)+\gamma \sqrt{\sum\limits_{i=1}^{n}\|X_{i}\|_{2}^{2}}.
\end{align*}
The desired result followed from Lemma 3.2 and the argument after it. \qed
\end{proof}

Clearly, Theorem \ref{4.4} degrades to Theorem \ref{4.2}, when $Gap(\lambda)=0$ and $\mathcal{B}=\{\theta^{*}(\lambda)\}$. If we want the rank of the solution of NRM (1) is no more than a given constant, Theorem \ref{4.4} implies that the regularization parameter $\lambda$ can be decided by feasible points of problems (2) and (3). Comparing to Theorem \ref{4.2}, this result is more easy to implement.

Actually, one can choose any feasible points of  problems (2) and (3) to obtain $Gap(\lambda)$. For instance, we give the following approach to obtain primal and dual feasible points. Let $\theta(\lambda)=\frac{\lambda\cdot \theta^{max}}{\lambda_{max}}$ with $\lambda\in(0,\lambda_{max}]$. Clearly, $\theta(\lambda)$ is  a feasible point of problem (3). Moreover, it is easy to check that
\begin{center}
$B(\lambda)=\frac{\sum\limits_{i=1}^{n}\theta_{i}(\lambda)X_{i}}{\lambda}=
\frac{\sum\limits_{i=1}^{n}\theta_{i}^{max}X_{i}}{\lambda_{max}}$,
\end{center}
 and $t_{i}(\lambda)=y_{i}-\langle X_{i},B(\lambda)\rangle$, ($i=1,2,\cdots,n$) are feasible points of problem (2). Clearly, these feasible points are related to $\theta^{max}$. Based on the definition of $\theta^{max}$, whether these feasible points can help to choose better regularization parameter depends on the specific form of $f_{i}$.

In next sections, we give two specific forms of $f_{i}$, which are least square function and Huber function. We show the efficient regularization parameter selection results based on this type feasible points.
 \section{Two specific applications: LS-NRM and H-NRM}
\label{sec:4}
Section 3 gives the novel regularization parameter selection rule for the general nuclear norm regularized minimization problem. This section illustrates the applications of this rule on nuclear norm regularized least square minimization (LS-NRM) and nuclear norm regularized Huber minimization (H-NRM).
\subsection{Regularization parameter selection rule for LS-NRM}
\label{sec:4.1}
It is well known that the nuclear norm regularized least square minimization (LS-NRM) is a popular method to recovery the low rank matrix. The model is given as
\begin{eqnarray}
\underset{B\in R^{p\times q}}\min\left\{\frac{1}{2}\sum\limits_{i=1}^{n}\left(y_{i}-\langle X_{i}, B\rangle\right)^{2}+\lambda\|B\|_{*}\right\},
\end{eqnarray}
Here, we use the notation $B^{*(ls)}(\lambda)$ as the solution of (7). Corresponding to problem (1), it is easy to show that $f_{i}(\mu)=\frac{1}{2} \mu^{2}$. Clearly, it is a differentiable function with 1-Lipschitz continuous gradient. Interestingly, the dual form of (7) is a projection problem on a convex area, which is given as
\begin{equation}\label{eq19}
\begin{split}
&\underset{\theta \in R^{n}}\max \left\{\langle \textbf{y}, \theta\rangle-\frac{1}{2}\|\theta\|_{2}^{2}\right\}\\
&s.t.\quad\left\|\sum\limits_{i=1}^{n} \theta_{i}X_{i}\right\|_{2}\leq \lambda.
\end{split}
\end{equation}
Let
$\theta^{max(ls)}=\underset{\theta \in R^{n}}{\arg\max}\left\{\langle \textbf{y}, \theta\rangle-\frac{1}{2}\|\theta\|_{2}^{2}\right\}=\textbf{y}$ and $\lambda^{(ls)}_{max}=\left\|\sum\limits_{i=1}^{n}y_{i}X_{i}\right\|_{2}$.
From Lemma 3.1, $\lambda^{(ls)}_{max}$ is the lower bound of the regularization parameter such that the solution of problem (7) is zero, and $\theta^{max(ls)}$ is its dual solution with the regularization parameter $\lambda^{(ls)}_{max}$. Then, for any $\lambda\in (0,\lambda^{(ls)}_{max}]$, we can easily give the feasible points of problems (7) and (8).
\begin{center}
 $B^{(ls)}_{0}(\lambda)=\frac{\sum\limits_{i=1}^{n}y_{i}X_{i}}{\lambda^{(ls)}_{max}}$,
 $\theta^{(ls)}_{0}(\lambda)=\frac{\lambda\textbf{y}}{\lambda^{(ls)}_{max}}$.
 \end{center}
Therefore, the dual solution is contained in the set $$\left\{\theta\Big|\|\theta-\theta^{(ls)}_{0}(\lambda)\|_{2}\leq \sqrt{\frac{2Gap^{(ls)}(\lambda)}{n}}\right\},$$
where
$Gap^{(ls)}(\lambda)$ is
\begin{center}
$\frac{1}{2}\sum\limits_{i=1}^{n}\left(y_{i}-\frac{\langle X_{i}, \sum\limits_{i=1}^{n}y_{i}X_{i}\rangle}{\lambda^{(ls)}_{max}}\right)^{2}
+\frac{\lambda}{\lambda^{(ls)}_{max}}\left\|\sum\limits_{i=1}^{n}y_{i}X_{i}\right\|_{*}
-\left(\frac{\lambda}{\lambda^{(ls)}_{max}}-\frac{1}{2}(\frac{\lambda}{\lambda^{(ls)}_{max}})^{2}\right)\|\textbf{y}\|_{2}^{2}.$
\end{center}

In order to show the closed form of the regularization parameters,  we define some new notations.
\begin{equation*}
\begin{aligned}
&&\lambda^{(ls)}_{0}&=\lambda^{(ls)}_{max},\\ &&a_{k}^{(ls)}&=n\frac{\left(\lambda^{(ls)}_{max}-\sigma_{k}
\left(\sum\limits_{i=1}^{n}y_{i}X_{i}\right)\right)^2}{\left(\lambda^{(ls)}_{max}\right)^{2}\cdot\sum\limits_{i=1}^{n}{\|X_{i}\|^{2}_{2}}}
-\left(\frac{\|\textbf{y}\|_{2}}{\lambda^{(ls)}_{max}}\right)^{2},\\ &&b^{(ls)}&=\frac{\left\|\sum\limits_{i=1}^{n}y_{i}X_{i}\right\|_{*}-\|\textbf{y}\|_{2}^{2}}{\lambda^{(ls)}_{max}},\\ &&c^{(ls)}&=\sum\limits_{i=1}^{n}\left(y_{i}-\frac{1}{\lambda^{(ls)}_{max}}\left\langle X_{i},\sum\limits_{i=1}^{n}y_{i}X_{i}\right\rangle\right)^{2},\\
&&d^{(ls)}&=\lambda^{(ls)}_{max}-\frac{\|\textbf{y}\|_{2}\cdot \sqrt{\sum\limits_{i=1}^{n}\|X_{i}\|_{2}^{2}}}{\sqrt{n}}\\
&&\Delta^{(ls)}_{k}&=\begin{cases}\sqrt{(b^{(ls)})^{2}+a_{k}^{(ls)}c^{(ls)}}, &\sqrt{(b^{(ls)})^{2}+a_{k}^{(ls)}c^{(ls)}}\geq0\\
\emptyset, &\rm{otherwise}.\end{cases}
\end{aligned}
\end{equation*}
We are ready to give the regularization parameter selection rule for LS-NRM.
\begin{Theorem}\label{4.5}
 Let $k \in \{1, 2,\cdots,r\}$. If $\lambda\in\left(\lambda^{(ls)}_{k},\lambda^{(ls)}_{k-1}\right]$, the solution of LS-NRM (7) satisfies that
 $\rm{rank}$$\left(B^{*(ls)}(\lambda)\right)\leq k-1$.
Here, the sequence of regularization parameters $\{\lambda^{(ls)}_{k}\}_{k=1}^{r}$ is given as
 \begin{equation}
\lambda^{(ls)}_{k}=
\begin{cases}
\max\left\{0,\frac{b^{(ls)}+\Delta^{(ls)}_{k}}{a_{k}^{(ls)}}\right\}, &~ \sigma_{k}\left(\sum\limits_{i=1}^{n}y_{i}X_{i}\right)<d^{(ls)}\\
\begin{cases}
\frac{c^{(ls)}}{-2b^{(ls)}}, &b^{(ls)}<0,\\
\emptyset, &\rm{otherwise},
\end{cases} & ~\sigma_{k}\left(\sum\limits_{i=1}^{n}y_{i}X_{i}\right)=d^{(ls)}\\
\left(\lambda^{l^{(1)}}_{k},\lambda^{u^{(1)}}_{k}\right)\cap(0,+\infty),
&~ \sigma_{k}\left(\sum\limits_{i=1}^{n}y_{i}X_{i}\right)>d^{(ls)}
\end{cases}
\end{equation}
where $\lambda^{l^{(1)}}_{k}=\frac{b^{(ls)}+\Delta^{(ls)}_{k}}{a_{k}^{(ls)}}$, and
$\lambda^{u^{(1)}}_{k}=\frac{b^{(ls)}-\Delta^{(ls)}_{k}}{a_{k}^{(ls)}}$.
\end{Theorem}
\begin{proof}
From Theorem \ref{4.4}, we can get that if $\lambda$ satisfies that
\begin{center}
$\lambda>\frac{\lambda}{\lambda_{max}}\sigma_{k}\left(\sum\limits_{i=1}^{n} y_{i}X_{i}\right)
+\sqrt{\sum\limits_{i=1}^{n}\|X_{i}\|_{2}^{2}}\cdot\sqrt{\frac{2Gap^{(ls)}(\lambda)}{n}},$
\end{center}
then $\rm{rank}$$\left(B^{*(ls)}(\lambda)\right)\leq k-1$. We simplify the above condition as
$$a_{k}^{(ls)}\lambda^{2}-2b^{(ls)}\lambda-c^{(ls)}>0,$$
which is a quadratic inequality of $\lambda$. In order to solve it, we consider the following three cases.

(i) For $a_{k}^{(ls)}>0$, which equals to $\sigma_{k}\left(\sum\limits_{i=1}^{n}y_{i}X_{i}\right)<d^{(ls)}_{max}$,
 $\lambda$ needs to satisfy that $$\lambda>\max\left\{0,\frac{b^{(ls)}+\Delta^{(ls)}_{k}}{a_{k}^{(ls)}}\right\}.$$

(ii) For $a_{k}^{(ls)}=0$, i.e., $\sigma_{k}\left(\sum\limits_{i=1}^{n}y_{i}X_{i}\right)=d^{(ls)}$,
 $\lambda$ needs to satisfy that
$\lambda>
\frac{c^{(ls)}}{-2b^{(ls)}}$ when $b^{(ls)}<0$, otherwise, it is meaningless.

(iii) For $a_{k}^{(ls)}<0$, which equals to
$\sigma_{k}\left(\sum\limits_{i=1}^{n}y_{i}X_{i}\right)>d^{(ls)}$,
 $\lambda$ needs to satisfy that $\lambda>0$ and
\begin{center}
 $\frac{b^{(ls)}+\Delta^{(ls)}_{k}}{a_{k}^{(ls)}}<\lambda<\frac{b^{(ls)}-\Delta^{(ls)}_{k}}{a_{k}^{(ls)}}.$
\end{center} \qed
Combing the above arguments and Theorem 3.2, we get the desired result.
\end{proof}

Theorem \ref{4.5} mainly state that  we can get a sequence of regularization parameters of problem (7) and the rank of the solution is bounded in the corresponding intervals. Note that the formula of the regularization parameter $\lambda_{k}$ is computed by three cases in above theorem. It is worth to mention that $a^{(ls)}_{k}$ is a nondecrease function of $k$ and $a^{(ls)}_{k}=0$ is rarely happened in practice.

As in Section 2, the nuclear norm regularized least square minimization is an efficient tool when
the noise has zero mean and constant variance. However, we do not know the distribution of the noise. To deal with this case, the nuclear norm regularized Huber minimization (H-NRM) is becoming attractive recently. In the proceed section, we focus on the regularization parameter selection rule for H-NRM.
\subsection{Regularization parameter selection rule for H-NRM}
\label{sec:4.2}
 It is well known that the Huber loss function is very important in statistics and machine learning (see, e.g., Huber \cite{H73}, Sun et al. \cite{S17}, Elsener and Geer \cite{E18} and Chen et al.\cite{C}), which is defined as
\begin{eqnarray*}
h_{\kappa}(t)=\begin{cases}
\frac{1}{2}t^{2}, &|t|\leq\kappa\\
\kappa|t|-\frac{1}{2}\kappa^{2}, &|t|>\kappa.
\end{cases}
\end{eqnarray*}
Clearly, the Huber  function is differentiable and the gradient is 1-Lipschitz continuous. Moreover, we give the conjugate function of the Huber function,
\begin{eqnarray*}
h_{\kappa}^{*}(\xi)=\underset{t\in R}\max\left\{t\xi-h_{\kappa}(t)\right\}
=\begin{cases}
\frac{1}{2}\xi^{2}, &0\leq|\xi|\leq\kappa\\
+\infty, &|\xi|>\kappa.
\end{cases}
\end{eqnarray*}
The nuclear norm regularized Huber minimization (H-NRM) is given as
\begin{eqnarray}
\underset{B\in R^{p\times q}}\min\left\{\sum\limits_{i=1}^{n}h_{\kappa}(y_{i}-\langle X_{i},B\rangle)+\lambda\|B\|_{*}\right\},
\end{eqnarray}
The solution of problem (10) is denoted as $B^{*(h)}(\lambda)$. As a result, the dual form of problem (10) is
\begin{equation}\label{eq16}
\begin{split}
&\underset{\theta\in R^{n}}\max\left\{\langle \textbf{ y},\theta\rangle-\frac{1}{2}\|\theta\|_{2}^{2}\right\}\\
&s.t.\quad\left\|\sum\limits_{i=1}^{n} \theta_{i}X_{i}\right\|_{2}\leq \lambda,\\
&\quad\quad -\kappa\textbf{1}\leq \theta \leq \kappa\textbf{1}.
\end{split}
\end{equation}
In a similar way in Section 4.1 for LS-NRM, we will give feasible points of problems (10) and (11), and present the sequence of regularization parameters of H-NRM. Note that when the solution of problem (10) is zero, its dual solution $\theta^{max(h)}$ is
 $$\theta^{max(h)}=\underset{\theta\in R^{n}}{\arg\max}\left\{\langle \theta,\textbf{y}\rangle-\frac{1}{2}\|\theta\|_{2}^{2}\Big| -\kappa\leq\theta_{i}\leq\kappa,  i=1,2,\cdots,n.\right\}.$$
 The closed-form of $\theta_{i}^{max(h)}$ is easily yielded as
\begin{eqnarray*}
\theta_{i}^{max(h)}=\begin{cases}
y_{i}, &|y_{i}|\leq \kappa,\\
sgn(y_{i})\kappa, &|y_{i}|>\kappa, \quad  i=1,2,\cdots, n.
\end{cases}
\end{eqnarray*}
Let $\tau_{i}=y_{i}\cdot\delta_{|\cdot|\leq\kappa}(y_{i})
+sgn(y_{i})\cdot\kappa\cdot\delta_{|\cdot|>\kappa}(y_{i})$ and $\tau=(\tau_{1},\cdots,\tau_{n})^{T}$. Then $\theta^{max(h)}=\tau$ and the lower bound of regularization parameter that enforces the solution of problem (10) being zero is $$\lambda^{(h)}_{max}=\left\|\sum\limits_{i=1}^{n}{\theta_{i}^{max(h)}X_{i}}\right\|_{2}=\left\|\sum\limits_{i=1}^{n}{\tau_{i}X_{i}}\right\|_{2}. $$
For any $0<\lambda<\lambda^{(h)}_{max}$, the feasible points of problems (10) and (11) are set as follows.
 \begin{center}
 $\theta^{(h)}_{0}(\lambda)=\frac{\lambda \theta^{max(h)}}{\lambda^{(h)}_{max}}=\frac{\lambda\tau}{\lambda^{(h)}_{max}}$ and $B^{(h)}_{0}(\lambda)=\frac{\sum\limits_{i=1}^{n}\theta_{i}^{0(h)}(\lambda)X_{i}}{\lambda}
 =\frac{\sum\limits_{i=1}^{n}\tau_{i}X_{i}}{\lambda^{(h)}_{max}}.$
 \end{center}
Therefore, the dual solution is contained in the set $$\left\{\theta\Big|\|\theta-\theta^{(h)}_{0}(\lambda)\|_{2}\leq \sqrt{\frac{2Gap^{(h)}(\lambda)}{n}}\right\},$$
where $Gap^{(h)}(\lambda)$ is\\
$\sum\limits_{i=1}^{n}h_{\kappa}(y_{i}-\frac{\langle X_{i}, \sum_{i=1}^{n}\tau_{i}X_{i}\rangle}{\lambda^{(h)}_{max}})
+\frac{\lambda}{\lambda^{(h)}_{max}}\left\|\sum\limits_{i=1}^{n}\tau_{i}X_{i}\right\|_{*}
-(\frac{\lambda}{\lambda^{(h)}_{max}}\langle\textbf{y},\tau\rangle-\frac{1}{2}(\frac{\lambda}{\lambda^{(h)}_{max}})^{2}\|\tau\|_{2}^{2})$. Moreover, we define some new notations.
\begin{equation*}
 \begin{aligned}
\lambda^{(h)}_{0}&=\lambda^{(h)}_{max}, \\ a_{k}^{(h)}&=n\frac{\left(\lambda^{(h)}_{max}-\sigma_{k}\left(\sum\limits_{i=1}^{n}\tau_{i}X_{i}\right)\right)^2}{\left(\lambda^{(h)}_{max}\right)^{2}\cdot\sum\limits_{i=1}^{n}{\|X_{i}\|^{2}_{2}}}
-\left(\frac{\|\tau\|_{2}}{\lambda^{(h)}_{max}}\right)^{2}, \\ b^{(h)}&=\frac{\left\|\sum\limits_{i=1}^{n}\tau_{i}X_{i}\right\|_{*}-\|\tau\|_{2}^{2}}{\lambda^{(h)}_{max}}\\ c^{(h)}&=\sum\limits_{i=1}^{n}\left(\tau_{i}-\frac{1}{\lambda^{(h)}_{max}}\left\langle X_{i},\sum\limits_{i=1}^{n}\tau_{i}X_{i}\right\rangle\right)^{2},\\
d^{(h)}&=\lambda^{(h)}_{max}-\frac{\|\tau\|_{2}\cdot \sqrt{\sum\limits_{i=1}^{n}\|X_{i}\|_{2}^{2}}}{\sqrt{n}},\\
\Delta^{(h)}_{k}&=\begin{cases}\sqrt{(b^{(h)})^{2}+a_{k}^{(h)}c^{(h)}}, &\sqrt{(b^{(h)})^{2}+a_{k}^{(h)}c^{(h)}}\geq0\\
\emptyset, &\rm{otherwise}.\end{cases}
\end{aligned}
\end{equation*}
\begin{Theorem}
 Let $k \in \{1, 2,\cdots,r\}$. If  $\lambda\in\left(\lambda^{(h)}_{k},\lambda^{(h)}_{k-1}\right]$, the solution of H-NRM (10) satisfies that $\rm{rank}$$\left(B^{*(h)}(\lambda)\right)\leq k-1$.  Here, the sequence of regularization parameters $\{\lambda^{(h)}_{k}\}_{k=1}^{r}$ is given as
 \begin{equation*}
\lambda^{(h)}_{k}=
\begin{cases}
\max\left\{0,\frac{b^{(h)}+\Delta^{(h)}_{k}}{a_{k}^{(h)}}\right\}, &~ \sigma_{k}\left(\sum\limits_{i=1}^{n}\tau_{i}X_{i}\right)<d^{(h)}\\
\begin{cases}
\frac{c^{(h)}}{-2b^{(h)}}, &b^{(h)}<0,\\
\emptyset, &\rm{otherwise},
\end{cases} & ~\sigma_{k}\left(\sum\limits_{i=1}^{n}
\tau_{i}X_{i}\right)=d^{(h)}\\
\left(\lambda^{l^{(2)}}_{k},\lambda^{u^{(2)}}_{k}\right)\cap(0,+\infty),
&~ \sigma_{k}\left(\sum\limits_{i=1}^{n}\tau_{i}X_{i}\right)>d^{(h)}
\end{cases}
\end{equation*}
where $\lambda^{l^{(2)}}_{k}=\frac{b^{(h)}+\Delta^{(h)}_{k}}{a_{k}^{(h)}}$, and
$\lambda^{u^{(2)}}_{k}=\frac{b^{(h)}-\Delta^{(h)}_{k}}{a_{k}^{(h)}}$.
\end{Theorem}
%
\section{Numerical experiments}
\label{sec:5.1}
We will report some numerical results in this section, which illustrate how to select the regularization parameter $\lambda$ such that the rank of the solution being no more that a given constant. Here, we choose the signal shapes of  device 8-20 and device 4-20 from  Zhou and Li \cite{Z14}, see Figure 1.  The size of device 8-20 and device 4-20 is $64\times 64$, and the rank of these two shaped are 4 and 14, respectively. In the following, we will do the numerical experiments under different noise distributions. When the noise is normal distribution, we consider the nuclear norm regularized least square minimization (LS-NRM). When the noise is $t$ distribution, we consider the nuclear norm regularized Huber minimization (H-NRM).
\begin{figure}[ht]
\centering
\includegraphics[height=5cm]{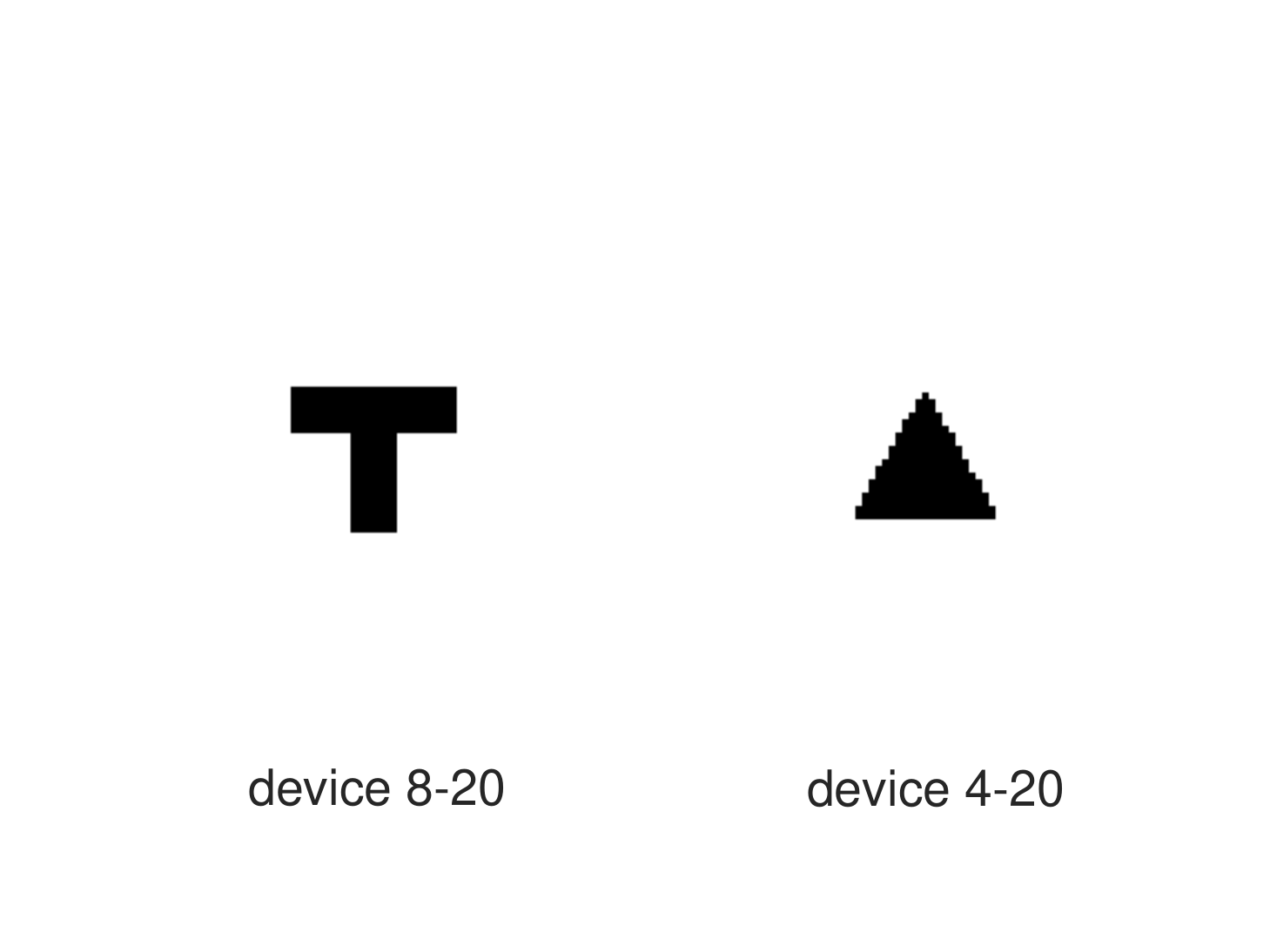}
\caption{Two signal shapes}
\end{figure}

\subsection {Normal distribution}
For every shape, we randomly simulate matrix $X_{i}$ $(i=1,2,\cdots,500)$ with elements obeying the standard Gaussian distribution, and the noise $\epsilon_{i}\sim N(0,0.1)$. Then, $y_{i}$ is obtained by $y_{i}=\langle X_{i}, B\rangle+\epsilon_{i}$, where $B$ is the numerical matrix converted by the shape and the elements of B  are 0 or 1.
\begin{figure}[ht]
\centering
\includegraphics[height=5.5cm]{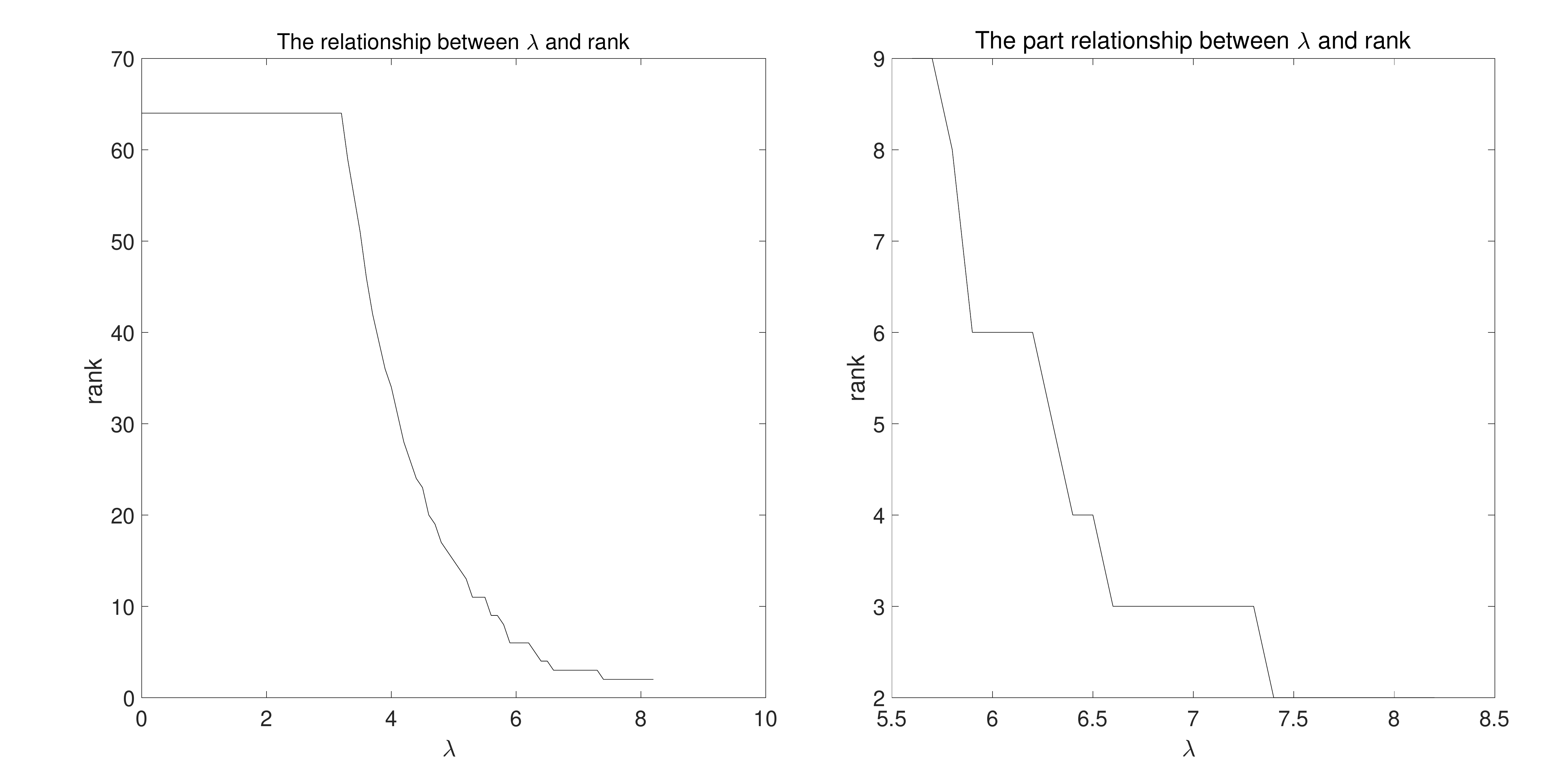}
\caption{Regularization parameter selection under Theorem 4.1 on signal device8-20}
\end{figure}

To recovery the given shapes by the nuclear norm regularized least square minimization problem (LS-NRM), we have to decide the best regularization parameter $\lambda$.  One nature method is cross validation which means that we need to try all $\lambda\in (0,\lambda^{(ls)}_{max}]$ and $\lambda^{(ls)}_{max}$ is defined in Section 4.1. In Figure 2, we show the relationship between $\lambda$ and the rank of the solution, on the shape device 8-20. From this figure, we know that the possible interval of the regularization parameter is $[6.4000,6.5000)$ where $\lambda^{(ls)}_{5}=6.4000$ and $\lambda^{(ls)}_{4}=6.5000$, rather than $(0,8.2043)$, where 8.2043 is the value of $\lambda^{(ls)}_{max}$ on this shape. Thus, under the rank of the solution is no more than 4, the regularization parameter selection result in Theorem 4.1 shrinks the possible interval of $\lambda$ from $[0,8.2043)$ to $[6.4000,6.5000)$.

Similarly, Figure 3 considers about the relationship between $\lambda$ and the solution rank on the shape device 4-20. Same as the analysis of Figure 2, the interval of the  regularization parameter can be shrink to $[5.2000,5.300)$ from $[0,8.1653)$, where $\lambda^{(ls)}_{15}=5.2000$, $\lambda^{(ls)}_{16}=5.3000$ and 8.1653 is the value of $\lambda^{(ls)}_{max}$ on this shape.
\begin{figure}[ht]
\centering
\includegraphics[height=5.5cm]{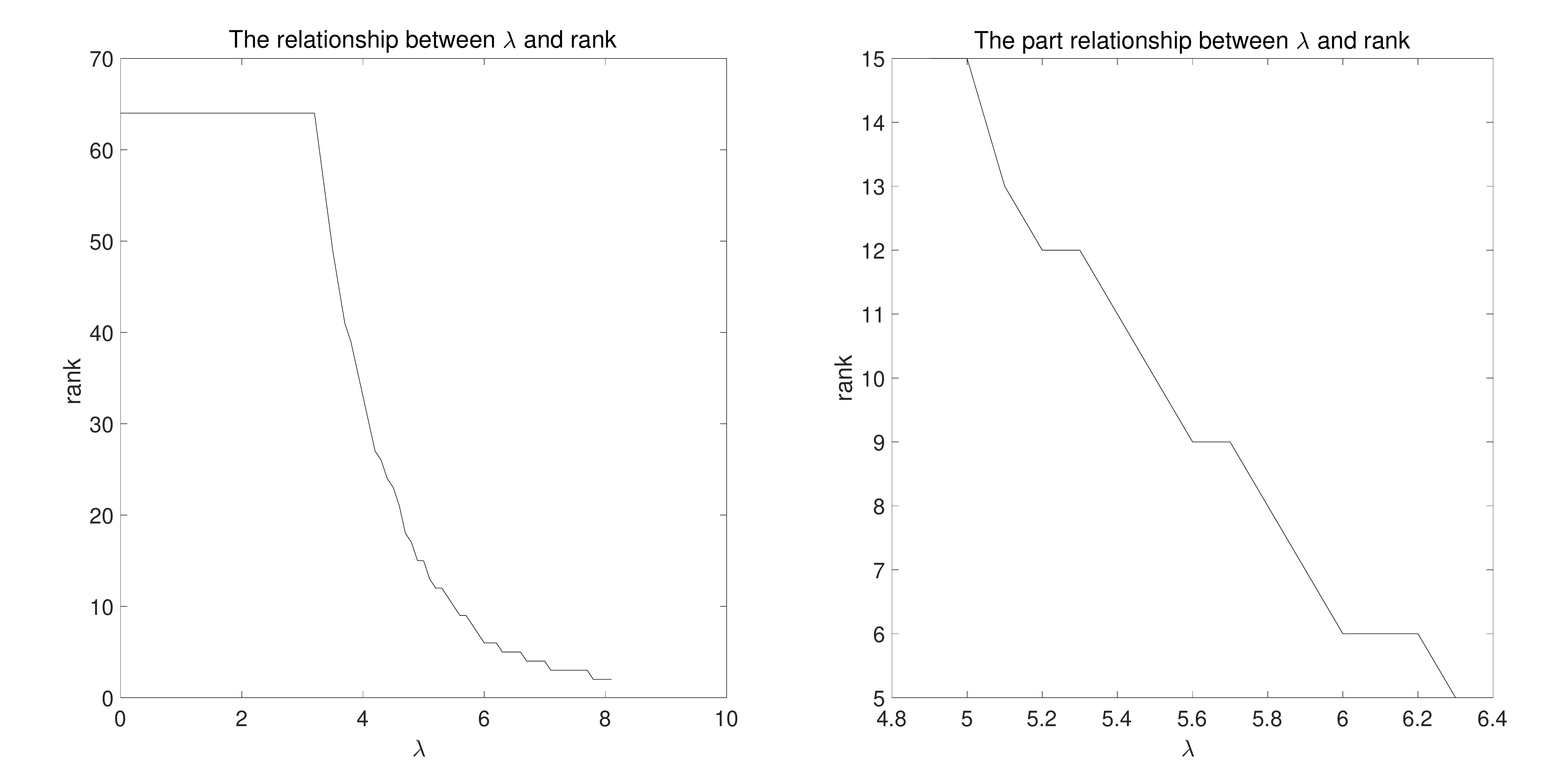}
\caption{Regularization parameter selection under Theorem 4.1 on signal device4-20}
\end{figure}

\subsection{$t$ distribution}
In order to present the efficiency of the regularization parameter selection rule for the nuclear norm regularized Huber minimization problem (H-NRM). Here, the parameter $\kappa$ of Huber function is 2.5. The way to simulate $\{X_{i},y_{i}\}_{i=1}^{500}$ is similar as above, except that the noise $\epsilon_{i}\sim t(3)$ as  in Elsener and Geer \cite{E18}. Figure 4 focuses on the shape device 8-20, the numerical result shows that $\lambda^{(h)}_{max}= 8.1615$, and we only need to consider the regularization parameter in interval $[6.6000,6.8000)$ when using the cross validation to decide the best choice of $\lambda$.  From the rank result of device 4-20 in Figure 5, $\lambda^{(h)}_{max}= 8.1292$ and the possible interval of the regularization parameter is shrunken to $[5.2000,5.1000)$.
\begin{figure}[ht]
\centering
\includegraphics[height=5.5cm]{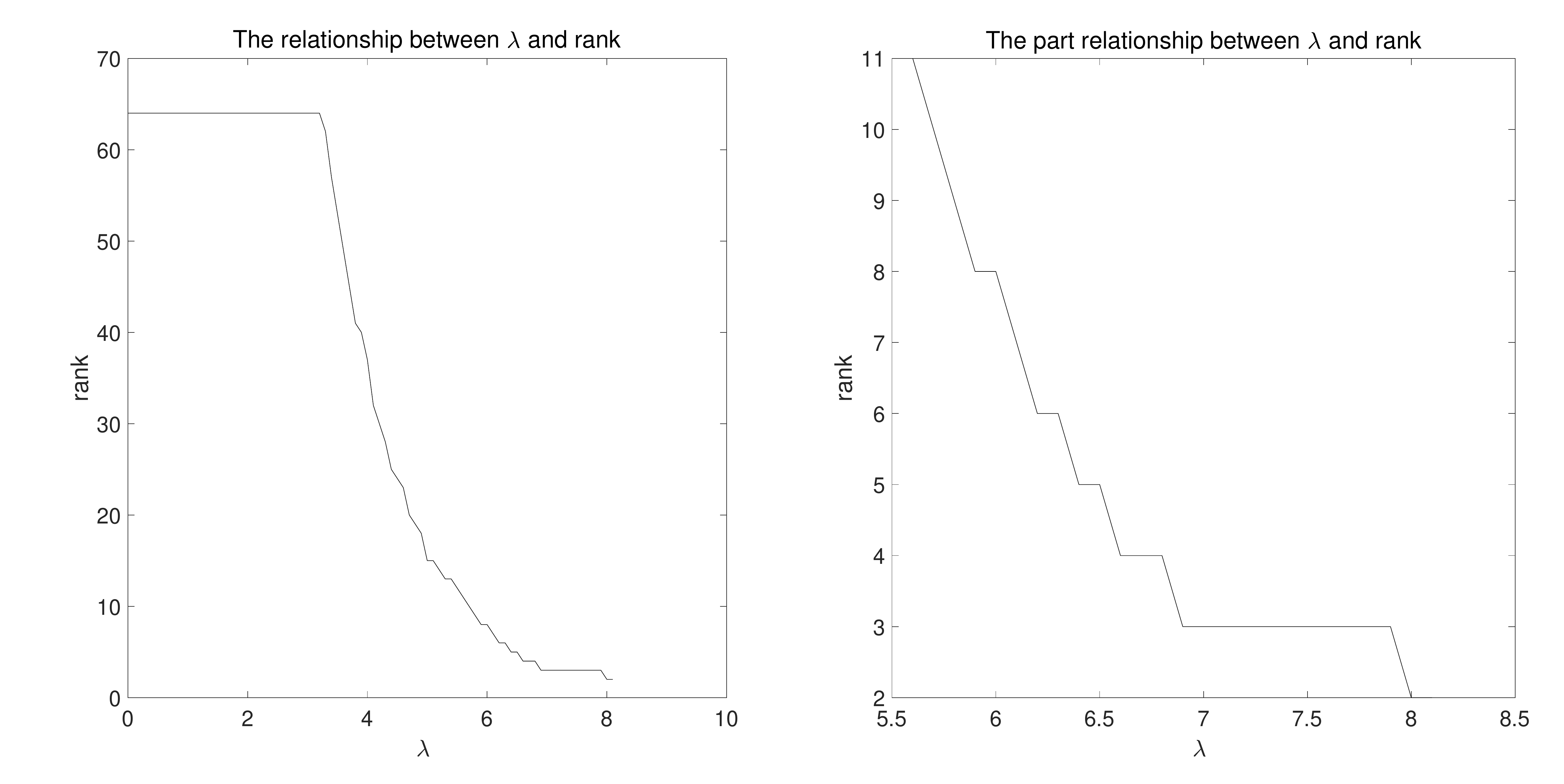}
\caption{Regularization parameter selection under Theorem 5.1 on signal device8-20}
\end{figure}

\begin{figure}[ht]
\centering
\includegraphics[height=5.5cm]{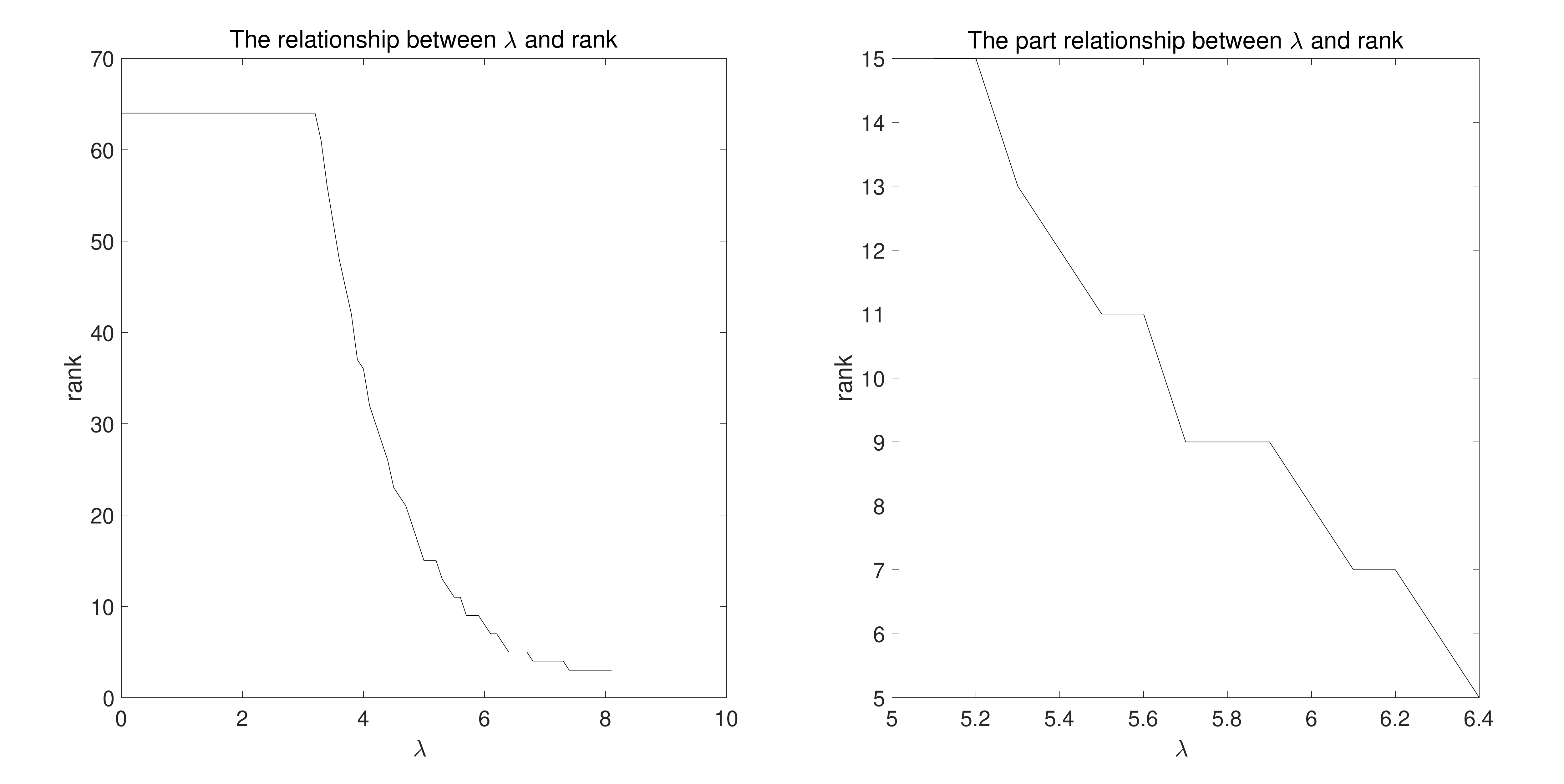}
\caption{Regularization parameter selection under Theorem 5.1 on signal device4-20}
\end{figure}
\section{Conclusion}
\label{sec:6}
With the help of duality theory, we propose a novel regularization parameter selection rule for the general nuclear norm regularized minimization problem (NRM), which is a popular model to recovery the low rank matrix. When we want the rank of the solution being no more than a constant, this rule helps to select the regularization parameter, which relates to  feasible points of NRM and its dual problem.
Moreover, this idea is applied to the nuclear norm regularized least square minimization and nuclear norm regularized Huber minimization, respectively. We get the sequence of closed-form regularization parameters of these two problems. Finally, we illustrate the regularization parameter selection rule on signal shapes, which indicate that our proposed rule can shrink the interval of the regularization parameter effectively.
\section*{Acknowledgements}
This work was supported by the National Science Foundation of China (11431002, 11671029).

\end{document}